\documentclass[10pt,leqno]{amsart}
\usepackage{amsmath}
\usepackage{amsfonts}
\usepackage{amssymb}
\usepackage{graphicx}
\usepackage{color}
\usepackage{hyperref}
\usepackage{xcolor}

\newtheorem{theorem}{Theorem}[section]
\newtheorem*{theorem*}{Theorem}
\newtheorem{lemma}{Lemma}[section]
\newtheorem{corollary}[theorem]{Corollary}

\newtheorem{remark}[theorem]{Remark}

\def\Ric{\text{Ric}}

\def\a{\alpha}
\def\l{\lambda}

\def\p{\partial}

\def\R{\mathbb{R}}

\def\vp{\varphi}

\def\k{\kappa}

\def\Ric{\operatorname{Ric}}

\numberwithin{equation}{section}

\begin{document}

\title[Eigenvalue of $p$-Laplacian on K\"ahler manifolds]{Lower bounds for the first eigenvalue of $p$-Laplacian on K\"ahler manifolds}

\author{Kui Wang} \thanks{The research of the first author is supported by NSFC No.11601359}
\address{School of Mathematical Sciences, Soochow University, Suzhou, 215006, China}
\email{kuiwang@suda.edu.cn}

\author{Shaoheng Zhang} \thanks{}
\address{School of Mathematical Sciences, Soochow University, Suzhou, 215006, China}
\email{20204207037@stu.suda.edu.cn}

%\date{\today}

%35P30, 35K55, 49R05, 58C40, 58J50, 53C26, 53C55}
%\subjclass[2010]{32Q15, 35P30, 35K55, 49R05, 58C40, 58J50, 53C26, 53C55}

\begin{abstract}
We study the eigenvalue problem for the $p$-Laplacian on K\"ahler manifolds. Our first result is a lower bound for the first nonzero eigenvalue of the $p$-Laplacian on compact K\"ahler manifolds  in terms of dimension, diameter, and lower bounds of holomorphic sectional curvature and orthogonal Ricci curvature for $p\in(1, 2]$. Our second result is a sharp lower bound for the first Dirichlet eigenvalue of the $p$-Laplacian on  compact  K\"ahler manifolds with smooth  boundary for $p\in(1,\infty)$. Our results generalize corresponding results for the Laplace eigenvalues on K\"ahler manifolds proved in \cite{LW21}.

\mbox{}

\end{abstract}

\maketitle

	\section{Introduction and Main Results}
Let $(M^m,g)$ be an $m$-dimensional compact  Riemannian manifold possibly with smooth boundary and the $p$-Laplacian operator  of the metric $g$ is defined by
	\[
	\Delta_p u:=\operatorname{div}(|\nabla u|^{p-2}\nabla u)
	\]
    for $p\in (1,\infty)$.
    The $p$-Laplacian eigenvalue equation is
    \begin{equation}
	\label{eqn plaplace equation}
		-\Delta_p u(x)=\lambda |u(x)|^{p-2}u(x), \quad x\in M,
	\end{equation}
    and $\l\in \R$ is called a closed eigenvalue  when $\p M=\emptyset$, a Dirichlet eigenvalue when  $\p M\neq \emptyset$ and $u(x)=0$ on $\p M$, a Neumann eigenvalue when  $\p M\neq \emptyset$ and $\p_\nu u(x)=0$ on $\p M$. Where $\nu$ denotes the outer unit normal to $\p M$.

The purpose of the present  paper  is to study the first nonzero eigenvalues of the $p$-Laplacian
on compact K\"ahler manifolds. In a recent paper \cite{LW21}, Li and the first author  obtained a lower bound for the first nonzero closed and Neumann eigenvalue (when $\p M\neq \emptyset$)
of the Laplacian on compact K\"ahler manifolds in terms of dimension, diameter, and lower bounds of holomorphic sectional curvature and orthogonal Ricci curvature. Precisely, they proved the following theorem.
\begin{theorem}[\cite{LW21}]
	\label{thm 1.1}
	    Let $\left(M^m, g, J\right)$ be a compact K\"ahler manifold of complex dimension $m$ and diameter $D$, whose holomorphic sectional curvature is bounded from below by $4 \kappa_1$ and orthogonal Ricci curvature is bounded from below by $2(m-1) \kappa_2$ for some $\kappa_1, \kappa_2 \in \mathbb{R}$. Let $\mu_{1,2}(M)$ be the first nonzero eigenvalue of the Laplacian on $M$ (with Neumann boundary condition if $M$ has a strictly convex boundary). Then
	    \[
        \mu_{1,2}(M) \geq \bar{\mu}_1\left(m, \kappa_1, \kappa_2, D\right),
        \]
        where $\bar{\mu}_1\left(m, \kappa_1, \kappa_2, D\right)$ is the first Neumann eigenvalue of the one-dimensional eigenvalue problem
        \begin{align}\label{1.2}
        \varphi^{\prime \prime}-\left(2(m-1) T_{\kappa_2}+T_{4      \kappa_1}\right) \varphi^{\prime}=-\lambda \varphi
        \end{align}
        on $[-D/2, D/2]$. Here and in the rest of this paper, we denote the function $T_{\kappa}$ for $\kappa \in \mathbb{R}$ by
	\begin{equation}
		\label{eqn T kappa}
		T_{\kappa}(t) =
		\begin{cases}
			\sqrt{\kappa}
			\tan (\sqrt{\kappa}t), & \kappa >0,\\
			0, & \kappa =0,\\
			-\sqrt{-\kappa}
			\tanh(\sqrt{-\kappa}t), & \kappa <0.
		\end{cases}
	\end{equation}
	\end{theorem}
On Riemannian manifolds, the above theorem known as Zhong-Yang estimate is proved by Zhong and Yang \cite{ZY84} for the case $\Ric\ge 0$, and by Kr\"oger \cite{Kr92} for the case $\Ric\ge (n-1)\k$ for general $\k\in\R$. These works use the gradient estimates method initiated by Li \cite{Li79} and Li-Yau \cite{LY80}.
In 2013, Andrews and Clutterbuck \cite{AC13} gave a simple proof via the modulus of continuity estimates for the solutions to the heat equation, see also \cite{CW94} for a coupling method, and  \cite{ZW17} for an elliptic proof based on \cite{An15} and \cite{Ni13}. Recently, Valtorta \cite{Va12} and Naber-Valtorta \cite{NV14} established the Zhong-Yang estimate for the first nonzero eigenvalue for $p$-Laplacian on Riemannian manifolds.

On compact K\"ahler manifolds, Li and the first author  proved a lower bound of the first nonzero eigenvalue of Laplacian in terms of geometric data (see Theorem \ref{thm 1.1}), and  Rutkowski and Seto \cite{BS22} established an explicit lower bound by studying the one-dimensional ODE \eqref{1.2}. On noncompact K\"ahler manifolds, Li-Wang \cite{LW05} and Munteanu \cite{Mu09} obtained upper bounds of the first Laplace eigenvalue under  conditions of bisectional curvature and Ricci curvature respectively. On closed K\"ahler manifolds with positive Ricci curvature bound from below, Lichnerowicz \cite{Li58} obtained a optimal lower bound of the first nonzero eigenvalue of Laplacian by using Reilly formula, and generalized by Blacker and Seto \cite{BS19} to the first nonzero eigenvalue of  $p$-Laplacian for $p\ge 2$.
It is thus a natural question to study the first nonzero eigenvalue of $p$-Laplacian on K\"ahler manifolds for $p\in(1, 2]$. Regarding this, we prove
	\begin{theorem}
	\label{thm 1.2}
	Let $(M^{m},g,J)$ be a compact K\"{a}hler manifold of complex dimension $m$ and diameter $D$, whose holomorphic sectional curvature is bounded from below by $4 \kappa_1$ and orthogonal Ricci curvature is bounded from below by $2(m-1) \kappa_2$ for some $\kappa_1, \kappa_2 \in \mathbb{R}$.
	Let $\mu_{1,p}(M)$ be the first nonzero eigenvalue of the $p$-Laplacian on M (with Neumann boundary condition if $M$ has a strictly convex boundary). Assume $1<p\leq 2$, then
	\begin{align}
		\mu_{1,p}(M) \geq \bar{\mu}_{1}(m,p,\kappa_1,\kappa_2,D),
	\end{align}
	where $\bar{\mu}_{1}(m,p,\kappa_1,\kappa_2,D)$ is the first nonzero Neumann eigenvalue of the one-dimensional eigenvalue problem
	\begin{equation}
	\label{1.5}
		(p-1)|\varphi'|^{p-2}\varphi''
		-( 2(m-1)T_{\kappa_2} +T_{4\kappa_1} )|\varphi'|^{p-2}\varphi'
		= -\mu |\varphi|^{p-2} \varphi
	\end{equation}
	on interval $[-D/2,D/2]$.
	\end{theorem}
	
%\begin{equation}		\label{eqn def of Lvarphi}\mathfrak{L}\varphi=(p-1)|\varphi'|^{p-2}\varphi''-( 2(m-1)T_{\kappa_2} +T_{4\kappa_1} )|\varphi'|^{p-2}\varphi'.\end{equation}

	%The Divergence theorem implies all the eigenvalues under three boundary conditions above are nonnegative, besides, the first closed eigenvalue and the first Neumann eigenvalue is zero, the first Dirichlet eigenvalue is positive. Denote the first nonzero closed eigenvalue and the first nonzero Neumann eigenvalue by $\mu_{1,p}(M)$, and the first Dirichlet eigenvalue by $\lambda_{1,p}(M)$. We refer the reader to the books \cite{MR1333601,MR768584,MR1320504} for classical results of eigenvalue problems.
	
	%In this article, we give the lower bound for the first nonzero eigenvalue of the $p$-Laplacian under three boundary conditions on K\"{a}hler manifolds. Our results will be based on a recent result of Li and Wang \cite{MR4328692}, which gives a comparison result of the first eigenvalue of the Laplacian on compact K\"{a}hler manifolds with a one-dimensional eigenvalue problem.
	%by using the modulus of continuity method developed by Ben Andrews and Julie Clutterbuck.

	%The following theorem is a lower bound for first nonzero eigenvalue of Laplacian in terms of dimension, diameter and Ricci curvature on K\"{a}hler manifold.
	
	%In \cite{LW21}, Li and Wang proved the following theorem, a lower bound for first nonzero eigenvalue of Laplacian in terms of dimension, diameter and holomorphic sectional curvature and orthogonal Ricci curvature on K\"{a}hler manifolds.

Note that if $\kappa_1=\kappa_2=0$,
	ODE \eqref{1.5} can be solved by $p$-trigonometric functions, then as a direct consequence of Theorem \ref{thm 1.2}, we have
	\begin{corollary}
		With the same assumptions as in Theorem \ref{thm 1.2}, and assume $\k_1=\k_2=0$. Then
		\begin{equation}
			\label{1.6}
			\mu_{1,p}(M) \geq (p-1)(\frac{\pi_p}{D})^{p},
		\end{equation}
		where $\displaystyle \pi_p=\frac{2 \pi}{p \sin(\pi /p)}$.
	\end{corollary}

	Now let us turn to lower bounds for Dirichlet eigenvalues on compact K\"ahler manifolds with smooth boundary. For  $\kappa, \Lambda \in \mathbb{R}$, denote by $C_{\kappa,\Lambda }$ the unique solution of the initial value problem
\begin{equation}
	\label{1.7}
		\left \{
		\begin{aligned}
			&\phi''(t)+\kappa \phi(t)=0,\\
			&\phi(0)=1,\ \phi'(0)=-\Lambda,
		\end{aligned}
		\right.
	\end{equation}
and denote $T_{\kappa,\Lambda}$ for
	$\kappa, \Lambda \in \mathbb{R}$ by
	\begin{equation}
	\label{1.8}
	    T_{\kappa,\Lambda}(t):
	    =-\frac{ C_{\kappa,\Lambda}'(t) }
	{ C_{\kappa,\Lambda}(t)}.	
	\end{equation}
For the first Dirichlet eigenvalue on compact K\"ahler manifolds, Li and the first author \cite[Theorem 1.4]{LW21} obtained a lower bound of the first Dirichlet eigenvalue of Laplacian using Barta's inequality, and Blacker-Seto \cite{BS19} proved a lower bound of the first Dirichlet eigenvalue of $p$-Laplacian for $p\geq2$ via the $p$-Reilly formula. Our second main result is the following $p$-Laplacian analogue of the lower bounds for the first Dirichlet eigenvalue on K\"ahler manifolds for $1<p<\infty$.
	
	\begin{theorem}
	\label{thm 1.3}
		Let $(M^m, g, J)$ be a compact K\"{a}hler manifold with smooth  boundary $\partial M$.
		Suppose that the holomorphic sectional curvature is bounded from below by $4\kappa_1$ and the orthogonal Ricci curvature is bounded from below by $2(m-1)\kappa_2$ for $\kappa_1,\kappa_2 \in \mathbb{R}$, and the second fundamental form on $\partial M$ is bounded from below by $\Lambda \in \mathbb{R}$.
		Let $\lambda_{1,p}(M)$ be the first Dirichlet eigenvalue of the $p$-Laplacian on $M$.
		Assume $1<p<\infty$, then
		\begin{align}\label{1.9}
		\lambda_{1,p}(M) \geq \bar{\lambda}_1(m,p,\kappa_1,\kappa_2,D),
		\end{align}
		where $\bar{\lambda}_1(m,p,\kappa_1,\kappa_2,D)$ is the first eigenvalue of the one-dimensional eigenvalue problem
\begin{align}\label{1.10}
		(p-1)|\vp'|^{p-2}\vp''
		- \big(2(m-1)T_{\kappa_2,\Lambda}
		+T_{4\kappa_1,\Lambda}\big)|\vp'|^{p-2} \vp' = -\l|\varphi|^{p-2} \varphi	
\end{align}
on $[0, D/2]$ with boundary conditions $\varphi(0)=0$ and $\varphi'(R)=0$.
	\end{theorem}
 The proof of Theorem \ref{thm 1.3}  relies on a comparison theorem for the second derivatives of distance to the boundary proved in \cite[Section 6]{LW21} and Barta's
inequality. 
In the Riemannian case, lower bounds for the first Dirichlet eigenvalue were proved by Li-Yau \cite{LY80} and by Kause \cite{Ka84}. It remains an interesting question that  whether or not the same result as Theorem \ref{thm 1.2} holds for $p>2$ and the result is sharp. We plan to return to this in the future. The rest of the paper is organized as follows. In Section 2, we recall the definitions of curvatures of K\"ahler manifolds and modulus of continuity of real functions.
Sections 3 and 4 are devoted to proving  Theorem \ref{thm 1.2} and Theorem \ref{thm 1.3}.
	
	%On K\"{a}hler manifold,Li and Wang \cite{MR4328692} obtain the lower bound of the first Dirichlet eigenvalue of Laplacian on K\"{a}hler manifold.Blacker and Seto \cite{MR3937693} prove the lower bound of the first Dirichlet eigenvalue of $p$-Laplacian for $p\geq2$ via the $p$-Reilly formula.

	\section{Preliminary}
	
	\subsection{Curvatures of K\"{a}hler Manifolds}
	
	Let $(M^m,g,J)$ be a K\"{a}hler manifold with complex dimension $m$ (real dimension is $n=2m$).
	A plane $\sigma \subset T_p M$ is said to be holomorphic if it is invariant by the complex structure tensor $J$.
	The restriction of the sectional curvature to holomorphic planes is called the holomorphic sectional curvature, denoted by $H$.
	In other words, if $\sigma$ is a holomorphic plane spanned by $X$ and $JX$, then the holomorphic sectional curvature of $\sigma$ is defined by
	\[
	H(\sigma):=H(X)=\frac{R(X,J X,X,J X)}{|X|^4}.
	\]
	We say the holomorphic sectional curvature is bounded from below by $\kappa$, if $H(\sigma) \geq \kappa$ for all holomorphic planes $\sigma \subset T_pM$ and all $p \in M$. The orthogonal Ricci curvature, denoted by $\operatorname{Ric}^{\perp}$, is defined for any $X\in T_pM$ by
	\[
	\operatorname{Ric}^{\perp} (X,X)
	:=\operatorname{Ric}(X,X)-H(X)|X|^2.
	\]
	We say the orthogonal Ricci curvature is bounded from below by $\kappa\in \mathbb{R}$,
	if
	$\operatorname{Ric}^{\perp}(X,X) \geq \kappa g(X,X),\forall X \in T_pM, \forall p \in M$,
	
	\begin{remark}
		\label{rmk 2.1}
		If $M$ is a complete K\"ahler manifold with holomorphic sectional curvature satisfies $H \geq \kappa_1$ for some positive $\k_1$, Tsukamoto \cite{Tsu57}  proved  that the diameter of $M$ is bounded from above by $\pi/\sqrt{\kappa_1}$. If
        $M$ is a complete K\"ahler manifold with orthogonal Ricci curvature satisfies $\operatorname{Ric}^\perp \geq 2(m-1)\kappa_2$  for some positive $\k_2$,
        Ni and Zheng \cite{NZ18} proved that the diameter of $M$ is bounded from above by $\pi/\sqrt{\k_2}$.
	\end{remark}
	
\subsection{Modulus of Continuity}
	
	Let $u$ be a continuous function on a metric space $(X,d)$, define the modulus of continuity $\omega$ of $u$ by
	\[
	\omega(s):=
	\sup\big\{ \frac{u(x)-u(y)}{2}|\ d(x,y)=2s, x, y\in X \big\}.
	\]
	Recently, Andrews and Clutterbuck \cite{An15,AC09,AC11,AC13} investigated how the modulus of continuity of  solutions to parabolic differential equations evolves,
	and they proved for a large class of parabolic equation, the modulus of continuity of the solution is a viscosity solution of the associated one-dimensional equations.
	By using the modulus of continuity,
	Andrews and Clutterbuck obtained the sharp lower bound on the fundamental gap for Schr\"odinger operators \cite{AC11}.
	This technique is called modulus of continuity estimate, and is widely used to study the lower bound of the first nonzero eigenvalue in terms of geometric data such as the diameter and dimension of the manifold, see \cite{HWZ20, LW21-d, SWW19} and so on.
	%We refer the reader to the surveys \cite{MR3381494} by Professor Andrews.
	
	\section{The First Nonzero Eigenvalue}
	
	In this section, we shall use the method of modulus of continuity estimates  to prove Theorem \ref{thm 1.2}. The proof presented below is a modification of the argument outlined in the survey by Andrews \cite[Section 8]{An15} for the case of Riemannian manifolds.
Let $\vp(s,t)\in C^{2,1}([0,a]\times[0, \infty))$, and denote by
	\begin{align}
	\mathfrak{L}\varphi
	:=(p-1)|\varphi'|^{p-2}\varphi''
	-( 2(m-1)T_{\kappa_2} +T_{4\kappa_1} )|\varphi'|^{p-2}\varphi',
	\end{align}
	where  we denote  $\frac{\p^k \vp}{\p s^k}$ by $\vp^{(k)}$ for short.
	We first prove the following modulus of continuity estimates for solutions to a nonlinear parabolic equation on K\"{a}hler manifolds.
	\begin{theorem}
	\label{thm 3.1}
		Let $(M^m, g, J)$ be a compact K\"ahler manifold with diameter $D$ whose holomorphic sectional curvature is bounded from below by $4\kappa_1$ and the orthogonal Ricci curvature is bounded from below by $2(m-1)\kappa_2$ for some $\kappa_1,\kappa_2 \in \mathbb{R}$.
		Let $v : M \times [0, T) \rightarrow \mathbb{R}$ be a solution of
		\begin{equation}
			\label{3.2}
			v_t
			=|\Delta_p v|^{-\frac{p-2}{p-1}}\Delta_p v
		\end{equation}
		(with Neumann boundary condition if $M$ has a strictly convex boundary).
		Suppose $1<p\le 2$ and
		$\varphi(s,t): [0,\frac{D}{2}] \times [0,\infty)\rightarrow \mathbb{R}$ satisfies
		\begin{itemize}
			\item[(1)] $v(y,0)-v(x,0) \leq 2 \varphi(\frac{d(x,y)}{2},0), \ \forall x,y \in M$,
			\item[(2)] $0 \geq \varphi_t \geq |\mathfrak{L}\varphi|^{-\frac{p-2}{p-1}} \mathfrak{L}\varphi,\
			(s,t) \in [0,D/2]\times \mathbb{R}_+$,
			\item[(3)] $\varphi'(s,t)>0,\
			(s,t) \in [0,D/2]\times \mathbb{R}_+$,
			\item[(4)] $\varphi(0,t) \geq 0,\ t >0 $.
		\end{itemize}
		Then
		\begin{align}\label{3.3}
		v(y,t)-v(x,t) \leq 2 \varphi(\frac{d(x,y)}{2},t)
		\end{align}
  for $t>0$ and $x,y\in M$.
	\end{theorem}
	
	%Next, we state and prove a lemma that will be used in the proof.
	
    \begin{proof}
		For $ \varepsilon >0$, let
		\begin{align*}
		A_{\varepsilon}(x,y,t)
		=v(y,t)-v(x,t)-2\varphi(\frac{d(x,y)}{2},t)
		-\varepsilon e^t.
		\end{align*}
		To prove \eqref{3.3}, it suffices to show that $A_{\varepsilon}<0$ for any $\varepsilon>0$.

  We argue by contradiction and assume that there exists  $(x_0,y_0,t_0)$ such that $A_{\varepsilon}$ attains its maximum zero on
		$M \times M \times [0,t_0]$ at $(x_0,y_0,t_0)$.
		Clearly $x_0 \neq y_0$, and $t_0 >0$. If $\p M\neq \emptyset$,
		similarly as in \cite[Theorem 3.1]{LW21},
		the strictly convexity of boundary, the Neumann condition and the positivity of $\varphi'$ rule out the possibility that $x_0 \in \partial M$ and $y_0 \in \partial M$.
		
		Compactness of $M$ implies that there exists an arc-length minimizing geodesic $\gamma_0$ connecting $x_0$ and $y_0$ such that $\gamma_0(-s_0)=x_0$ and $\gamma_0(s_0)=y_0$ with $s_0=d(x_0,y_0)/2$.
		We pick an orthonormal basis $ \{ e_i(s) \}_{i=1}^{2m}$ for $T_{x_0}M$ with $e_1 = \gamma'_0(-s_0)$ and $e_2=J\gamma_0'(-s_0)$, where $J$ is the complex structure.
		Then parallel transport along $\gamma_0$ produces an orthonormal basis $ \{ e_i(s) \}_{i=1}^{2m}$ for $T_{\gamma_0(s)}M$ with $e_1(s)=\gamma_0'(s)$ for each $s\in [-s_0, s_0]$.
		Since $J$ is parallel and $\gamma_0$ is a geodesic, we have $e_2(s)=J\gamma_0'(s)$ for each $s\in [-s_0, s_0]$.

     First derivative inequality  yields
		\begin{equation}
			\label{3.4}
			0 \leq \partial_t
			A_{\varepsilon}(x_0,y_0,t_0)
			=v_t (y_0,t_0)-v_t(x_0,t_0)
			-2\varphi_t - \varepsilon e^{t_0},
		\end{equation}
		and
		$$
		\nabla_x A_{\varepsilon} (x_0,y_0,t_0)
		=\nabla_y A_{\varepsilon} (x_0,y_0,t_0)=0,
		$$
	namely
		\begin{equation}
			\label{3.5}
			\nabla v(x_0,t_0)=\varphi'(s_0,t_0)e_1(-s_0),
			\nabla v(y_0,t_0)=\varphi'(s_0,t_0)e_1(s_0).
		\end{equation}
		Recall that
		\begin{equation}
			\label{3.6}
			\Delta_p v
			= |\nabla v|^{p-2} \Delta v
			+ (p-2)|\nabla v|^{p-4} \nabla^2v(\nabla v,\nabla v),
		\end{equation}
		then plugging  equality \eqref{3.5} into \eqref{3.6} we get
		\begin{align}\label{3.7}
	\begin{split}			
    &\Delta_p v(y_0,t_0)-\Delta_p v(x_0,t_0)\\
				=& (p-1) |\varphi'|^{p-2}
				(v_{11}(y_0,t_0)-v_{11}(x_0,t_0) )+|\varphi'|^{p-2} \sum_{i=2}^{2m}
				(v_{ii}(y_0,t_0)-v_{ii}(x_0,t_0) ).
    \end{split}
		\end{align}
		
  Now we use the first and second variation formulas of arc length to calculate the second derivatives in space variables.
		Suppose
		$\gamma(r, s) :
		[-\delta,\delta]\times [-s_0,s_0] \rightarrow M$ is a smooth variation of $\gamma_0$ such that $\gamma(0,s)=\gamma_0(s)$, then the variation formulas give
		\begin{align}\label{3.8}
		\frac{ \mathrm{d} }{ \mathrm{d} r} \big|_{r=0}
		L[\gamma(r,s)]
		=
		g( T,\gamma_r ) \big|_{-s_0}^{s_0},
		\end{align}
		and
		\begin{align}\label{3.9}
		\frac{ \mathrm{d}^2}{ \mathrm{d} r^2} \big|_{r=0}
		L[\gamma(r,s)]
		=
		\int_{-s_0}^{s_0}
		( | (\nabla_s \gamma_r)^{\perp} |^2-R(T,\gamma_r,T,\gamma_r) ) \ \mathrm{d} s
		+g( T,\nabla_r \gamma_r )\big|_{-s_0}^{s_0},
		\end{align}
		where $T=\p_s \gamma(0,s)$ is the unit tangent vector to $\gamma_0$.

		To calculate the second derivative along $e_1$, we consider the variation $\gamma(r,s):=\gamma_0(s+r\frac{s}{s_0})$, and $ T=\gamma_0'(s)$ and $\gamma_r=\gamma_0'(s)s/s_0$. Then formulas \eqref{3.8} and \eqref{3.9} give
		$$\frac{\mathrm{d}}{\mathrm{d} r} \big|_{r=0} L[\gamma(r)]=2,\quad\text{and}\quad \frac{\mathrm{d}^2}{\mathrm{d} r^2}
		\big|_{r=0} L[\gamma(r)] =0.$$
		Hence the second derivative test for this variation produces
		\begin{equation}
			\label{3.10} v_{11}(y_0,t_0)-v_{11}(x_0,t_0)-2 \varphi''(s_0,t_0) \leq 0.
		\end{equation}
		
		To calculate the second derivative along $e_2$, we
  denote by
  \begin{equation}
			\label{3.11}
			c_{\kappa}(t) =
			\begin{cases}
				\operatorname{cos}(\sqrt{\kappa}t), & \kappa >0,\\
				1, & \kappa =0,\\
				\operatorname{cosh}(\sqrt{-\kappa}t), & \kappa <0,
			\end{cases}
		\end{equation}
  and consider the variation $$\gamma(r,s):=\operatorname{exp}_{\gamma_0(s)}( r\eta(s)e_2(s) ),$$ where $\displaystyle \eta(s)=\frac{c_{4\kappa_1}(s)}{c_{4\kappa_1}(s_0)}$. Clearly $T =e_1(s)$ and $\gamma_r=\eta(s)e_2(s)$, and formulas \eqref{3.8} and \eqref{3.9} imply
		 $$\frac{\mathrm{d}}{\mathrm{d} r} \bigg|_{r=0} L[\gamma(r)]=0$$ and
		\[
		\frac{\mathrm{d}^2}{\mathrm{d} r^2} \bigg|_{r=0} L[\gamma(r)]
		= \int_{-s_0}^{s_0}
		\bigg( (\eta')^2 -\eta^2 R(e_1,e_2,e_1,e_2) \bigg) \
		\mathrm{d} s.
		\]
		So this variation produces
		\[
		v_{22}(y_0,t_0)-v_{22}(x_0,t_0)
		-\varphi'(s_0,t_0) \int_{-s_0}^{s_0}
		\bigg( (\eta')^2 -\eta^2 R(e_1,e_2,e_1,e_2) \bigg) \ \mathrm{d} s \leq 0.
		\]
Using the assumption that $H\ge 4\k_1$ and integration by parts, we estimate that
		\begin{align*}
			&\int_{-s_0}^{s_0}
			\bigg( (\eta')^2 -\eta^2 R(e_1,e_2,e_1,e_2) \bigg)
			\ \mathrm{d} s\\
			=&  \eta'\eta |_{-s_0}^{s_0}
			-\int_{-s_0}^{s_0}
			\eta^2 \big(R(e_1,e_2,e_1,e_2)-4\kappa_1\big)
			\ \mathrm{d} s\\
   	\le &  \eta'\eta |_{-s_0}^{s_0}
			-\int_{-s_0}^{s_0}
			\eta^2 \big(R(e_1,e_2,e_1,e_2)-4\kappa_1\big)
			\ \mathrm{d} s\\
   =& -2 T_{4\kappa_1}(s_0).
		\end{align*}
Therefore we conclude from the above two inequalities that
		\begin{equation}
			\label{3.12}
			v_{22}(y_0,t_0)-v_{22}(x_0,t_0)
			\leq
			-2 T_{4\kappa_1}(s_0) \varphi'(s_0,t_0),
		\end{equation}
where we used the assumption that $\vp'(s,t)>0$.
		
			To calculate the second derivative along $e_i
$($3\le i\le 2m$), we  consider the variation $$\gamma(r,s): = \operatorname{exp}_{\gamma_0(s)}( r\zeta(s)e_i(s) ),$$
		where $\zeta(s)=\frac{c_{\kappa_2}(s)}{c_{\kappa_2}(s_0)}$.
		Similarly, the second variation formula gives
		\[
		v_{ii}(y_0,t_0)-v_{ii}(x_0,t_0)
		-\varphi'(s_0,t_0)
		\int_{-s_0}^{s_0}
		\bigg( (\zeta')^2 -\zeta^2 R(e_1,e_i,e_1,e_i) \bigg) \ \mathrm{d} s \leq 0.
		\]
		Summing over $3\le i\le 2m$ yields
		\[
		\sum_{i=3}^{2m}
		\big(v_{ii}(y_0,t_0)-v_{ii}(x_0,t_0)\big)
		-\varphi'(s_0,t_0)
		\sum_{i=3}^{2m}
		\int_{-s_0}^{s_0}
		\bigg(  (\zeta')^2 -\zeta^2 R(e_1,e_i,e_1,e_i) \bigg) \ \mathrm{d} s \leq 0.
		\]
		Using integration by parts and assumption that $\Ric^\perp\ge 2(m-1)\k_2$, we have
		\begin{align*}
			&\sum_{i=3}^{2m}
			\int_{-s_0}^{s_0}
			\bigg( (\zeta')^2 -\zeta^2 R(e_1,e_i,e_1,e_i) \bigg) \mathrm{d} s\\
			=& -4(m-1) T_{\kappa_2}(s_0)
			-\int_{-s_0}^{s_0}
			\zeta^2
			\big(
			-2(m-1)\kappa_2
			+\sum_{i=3}^{2m} R(e_1,e_i,e_1,e_i)
			\big)
			\mathrm{d} s\\
   \le &  -4(m-1) T_{\kappa_2}(s_0),
		\end{align*}
		thus we get
		\begin{equation}
			\label{3.13}
			\sum_{i=3}^{2m}
			\big(v_{ii}(y_0,t_0)-v_{ii}(x_0,t_0)\big)
			\leq
			-4(m-1) T_{\kappa_2}(s_0) \varphi'(s_0,t_0).
		\end{equation}
		
		Inserting inequalities \eqref{3.10}, \eqref{3.12} and \eqref{3.13} into \eqref{3.7}, we get
		\begin{equation}
			\label{3.14}
			\Delta_p v(y_0,t_0)-\Delta_p v(x_0,t_0)
			\leq
			2\mathfrak{L} \varphi(s_0,t_0).
		\end{equation}
		Let $h(t):=|t|^{-\frac{p-2}{p-1}}t$, which is odd, increasing and convex for all $t>0$, then  \eqref{3.14} implies
		\[
		h(\Delta_p v(y_0,t_0))
		\leq
		h(\Delta_p v(x_0,t_0) +2\mathfrak{L} \varphi).
		\]
		Applying Lemma 2.1 of  \cite{LW21-c} with
		$t=\Delta_p v(x_0,t_0)$ and
		$\delta= -\mathfrak{L} \varphi\geq0$, we get
		\[
		h(\Delta_p v(x_0,t_0) +2\mathfrak{L} \varphi)
		\leq
		h(\Delta_p v(x_0,t_0))+2h(\mathfrak{L} \varphi),
		\]
		that is
		\begin{equation}
			\label{3.15}
			\frac{1}{2}\big[
			|\Delta_p v|^{-\frac{p-2}{p-1}}\Delta_p v \big|_{(y_0,t_0)}
			-|\Delta_p v|^{-\frac{p-2}{p-1}}\Delta_p v \big|_{(x_0,t_0)}
			\big]
			\leq  |\mathfrak{L}\varphi|^{-\frac{p-2}{p-1}}
			\mathfrak{L}\varphi.
		\end{equation}
		Combining \eqref{3.4} and \eqref{3.15}, we get
		\[
		\varphi_t \leq
		|\mathfrak{L}\varphi|^{-\frac{p-2}{p-1}}\mathfrak{L}\varphi
		-\frac{\varepsilon}{2} e^{t_0}
		< |\mathfrak{L}\varphi|^{-\frac{p-2}{p-1}}\mathfrak{L}\varphi,
		\]
		which contradicts the inequality in assumption (2), so
  $$A_\varepsilon(x,y,t)<0$$
  holds true for all $x,y\in M$ and $t>0$. Hence we complete the proof of Theorem \ref{thm 3.1}.
	\end{proof}	
	
  	On the interval $[0,D/2]$, we define the following corresponding one-dimensional eigenvalue problem
	\[
	\bar{\sigma}_1
	=
	\inf
	\bigg\{
	\frac
	{ \int_0^{\frac{D}{2}} |\phi'|^p c_{\kappa_2}^{2m-2} c_{4\kappa_1} \ \mathrm{d} s }
	{ \int_0^{\frac{D}{2}} |\phi|^p c_{\kappa_2}^{2m-2} c_{4\kappa_1} \ \mathrm{d} s }
	\ \big|\phi \in W^{1,p}\big((0,D/2)\big)\backslash\{0\},
	\  \phi(0)=0
	\bigg\},
	\]
	where $c_{\kappa}(t)$ is defined by \eqref{3.11} and $D$ is the diameter of $M$. Noting from  Remark \ref{rmk 2.1} that $D\le \pi/\sqrt{4\k_1}$ if $\k_1>0$, and $D\le \pi/\sqrt{\k_2}$ if $\k_2>0$.
	
	\begin{lemma}
	\label{lm 3.1}
		\[
		\bar{\mu}_1(m,p,\kappa_1,\kappa_2,D)
		=\bar{\sigma}_1(m,p,\kappa_1,\kappa_2,D).
		\]
	\end{lemma}
	\begin{proof}
		Suppose $\phi(s)$ is an eigenfunction on $[0,D/2]$ corresponding to $\bar{\sigma}_1(m,p,\k_1,\k_2,D)$, and define a test function for $\bar{\mu}_1(m,p,\kappa_1,\kappa_2,D)$ by
		\[
		\bar{\phi}(s)=
		\begin{cases}
			\phi(s), & s \in [0, D/2],\\
			-\phi(-s), &s \in [-D/2,0],
		\end{cases}
		\]
		then we get
		\[
		\bar{\mu}_1(m,p,\kappa_1,\kappa_2,D)
		\leq
		\bar{\sigma}_1(m,p,\kappa_1,\kappa_2,D).
		\]
		
		On the other hand, suppose $\psi(s)$ is an eigenfunction corresponding to $\bar{\mu}_1(m,p,\kappa_1,\kappa_2,D)$, and then
		direct calculation gives
		\[
		(|\psi'|^{p-2} \psi' c_{\kappa_2}^{2m-2} c_{4\kappa_1} )'
		=
		-c_{\kappa_2}^{2m-2} c_{4\kappa_1}
		\bar{\mu}_1 |\psi|^{p-2} \psi.
		\]
		Integrating  the above equation over $[-D/2,D/2]$ yields
		\[
		0=\int_{-\frac{D}{2}}^{\frac{D}{2}}
		c_{\kappa_2}^{2m-2} c_{4\kappa_1} |\psi|^{p-2} \psi
		\ \mathrm{d} s,
		\]
		so there exists $s_0 \in (-D/2,D/2)$ such that $\psi(s_0)=0$. Without loss of generality, we  assume further that $s_0 \in [0,D/2)$, and define $\psi(s)=0$ for $s\in [0,s_0)$. Then $\psi$ is a test function for  $\bar{\sigma}_1$, and we have
		\begin{align*}
			\bar{\sigma}_1(m,p,\kappa_1,\kappa_2,D)
			\leq &
			\frac
			{ \int_0^{\frac{D}{2}} |\psi'|^p c_{\kappa_2}^{2m-2} c_{4\kappa_1} \ \mathrm{d} s}
			{ \int_0^{\frac{D}{2}} |\psi|^p c_{\kappa_2}^{2m-2} c_{4\kappa_1} \ \mathrm{d} s }\\
			=&
			\frac
			{ \int_{s_0}^{\frac{D}{2}} |\psi'|^p c_{\kappa_2}^{2m-2} c_{4\kappa_1} \ \mathrm{d} s}
			{ \int_{s_0}^{\frac{D}{2}} |\psi|^p c_{\kappa_2}^{2m-2} c_{4\kappa_1} \ \mathrm{d} s}\\
			=& \bar{\mu}_1(m,p,\kappa_1,\kappa_2,D).
		\end{align*}
		Thus Lemma \ref{lm 3.1} holds.	
	\end{proof}
	
	\begin{lemma}
	\label{lm existence of odd function}
		There exists an odd eigenfunction $\phi$ corresponding to
		$\bar{\mu}_1(m,p,\kappa_1,\kappa_2,D)$
		satisfying
		\begin{align}\label{3.16}
		(p-1)|\phi'|^{p-2}\phi''
		-( 2(m-1)T_{\kappa_2} +T_{4\kappa_1} )|\phi'|^{p-2}\phi'
		= -\bar{\mu}_1(m,p,\kappa_1,\kappa_2,D)
		|\phi|^{p-2} \phi
		\end{align}
		on $[0,D/2]$ with $\phi(s)>0$ for $s\in (0,D/2]$, $\phi'(s) >0$ for $s \in [0,D/2)$, and $\phi'(D/2)=0$.
	\end{lemma}
	
	\begin{proof}
		From the proof of Lemma \ref{lm 3.1}, we see that
		there exists an odd function $\phi$ corresponding to $\bar{\mu}_1(m,p,\kappa_1,\kappa_2,D)$ with $\phi(0)=0$ and $\phi'(D/2)=0$, which does not change sign in $(0,D/2]$.
		Since equation \eqref{3.16} is equivalent to
		\[
		(|\phi'|^{p-2} \phi' c_{\kappa_2}^{2m-2} c_{4\kappa_1} )'
		=-\bar{\mu}_1
		c_{\kappa_2}^{2m-2} c_{4\kappa_1} |\phi|^{p-2} \phi
		\]
  and $\phi'(D/2)=0$, 
		then  $\phi'(s)>0$ on $[0,D/2)$.
	\end{proof}
	Now we turn to prove Theorem \ref{thm 1.2}.
	\begin{proof}[Proof of Theorem \ref{thm 1.2}]
		For any $D_1>D$, let $\bar{\mu}_1=\bar{\mu}_1(m,p,\kappa_1,\kappa_2,D_1)$
		be the first nonzero Neumann eigenvalue of the eigenvalue problem \eqref{1.5}.
		By Lemma \ref{lm existence of odd function}, there exists an odd eigenfunction $\phi(s)$ corresponding to $\bar{\mu}_1$ such that $\phi(s)>0$ on $(0,D_1/2]$,
		$\phi'(s) >0$ on $[0,D_1/2)$ and $\phi'(D_1/2)=0$. Let $u(x)$ be an eigenfucntion corresponding to $\mu_{1, p}(M)$, and it is well known that $u(x)\in C^{1,\a}(M)$ for some $\a\in(0,1)$, then
	 there exists $C_1>0$ such that
		\begin{equation}
		\label{eqn lower bound of eigenfunction u}
			|u(y)-u(x)|
			\leq C_1 d(x,y)
		\end{equation}
  for $x,y\in M$. Noting that  $\vp'$ is positive on $[0, D/2]$, then from equation \eqref{3.16} we deduce that
  $\vp$ is smooth on  $[0, D/2]$, and for all $x,y\in M$ it holds
		\begin{equation}
		\label{eqn lower bound of eigenfunction phi}
			\phi(\frac{d(x,y)}{2})
			\geq
			C_2 d(x,y)		
		\end{equation}
        for some $C_2>0$.
		By (\ref{eqn lower bound of eigenfunction u}) and (\ref{eqn lower bound of eigenfunction phi}), we can choose a $C>0$ such that
		\[
		|u(y)-u(x)| \leq 2C \phi(\frac{d(x,y)}{2})
		\]
		for all $x,y\in M$.

		Let
		\begin{align*}
		v(x,t)=\exp(-\mu_{1,p}^{\frac{1}{p-1}}t) u(x),
  \end{align*}
  and
  \begin{align*}
  \varphi(s,t)=C \exp(-\bar{\mu}_1^{\frac{1}{p-1}}t) \phi(s).
		\end{align*}
It is easy to verify that $v$ and $\varphi$ satisfy the conditions in Theorem \ref{thm 3.1}, and then it holds
		\[
		v(y,t)-v(x,t) \leq
		2 \varphi(\frac{d(x,y)}{2},t)
		\]
  for all $t>0$, namely
		\begin{equation*}
		\exp(-\mu_{1,p}^{\frac{1}{p-1}}t) \big(u(y)- u(x)\big)
		\leq
		2 C \exp(-\bar{\mu}_{1}^{\frac{1}{p-1}}t) \phi(\frac{d(x,y)}{2}).
		\end{equation*}
		As $t \rightarrow \infty$, the above inequality gives
		\[
		\mu_{1,p} \geq
		\bar{\mu}_1 (m,p,\kappa_1,\kappa_2,D_1),
		\]
		then Theorem \ref{thm 1.2} follows by letting $D \rightarrow D_1$.	
	\end{proof}

\section{The First Dirichlet Eigenvalue}
	In this section, we give the proof of Theorem \ref{thm 1.3}.
	The key tools are a comparison theorem for distance to the boundary and  Barta's inequality.
	Let $M$ be a compact manifold with smooth boundary $\partial M$, and define the distance function to the boundary of $M$ by
	\[
	d(x, \partial M)=\inf \Big \{ d(x,y): y \in \partial M \Big\}.
	\]
By choosing $\alpha=(p-1)|\nabla \psi|^{p-2}$ and $\beta=|\nabla \psi|^{p-2}$ in Theorem 6.1 of \cite{LW21}, we have the following lemma.
	\begin{lemma}
	\label{lemma main dirichlet eigenvalue}
		Let $(M^m, g, J)$ be a compact K\"ahler manifold with smooth  boundary $\partial M$.
	    Suppose that $H \geq 4\kappa_1$ and $\operatorname{Ric}^{\perp}\geq2(m-1)\kappa_2$ for some $\kappa_1,\kappa_2\in \mathbb{R}$, and the second fundamental form on $\partial M$ is bounded from below by $\Lambda \in \mathbb{R}$.
	  Let $p\in(1, \infty)$, and assume $\varphi$ is a smooth function on $[0,R]$ satisfying $\varphi'\geq 0$.
	    Then for any smooth function $\psi$ satisfying
	    \[
		\psi(x) \leq \varphi(d(x,\partial M)) \quad \text{for} \quad x \in M,\quad \text{and}\quad
		\psi(x_0) =\varphi(d(x_0,\partial M)),
		\]
	it holds
		\[
	    \Delta_p \psi (x_0)
	    \leq
	    \bar{\mathfrak{L}}\varphi(d(x_0,\partial M)),
	    \]
    where one-dimensional operator $\bar{\mathfrak{L}}$ is defined by
\begin{align}
	\bar{\mathfrak{L}}\varphi
	=(p-1)|\varphi'|^{p-2}\varphi''
	-\big(2(m-1)T_{\kappa_2,\Lambda}
	+T_{4\kappa_1,\Lambda}\big) |\varphi'|^{p-2} \varphi'
\end{align}
for all $\varphi\in C^2([0, R])$, and $T_{\kappa,\Lambda}$ is defined by \eqref{1.8}.
	\end{lemma}
	Consider the following  one-dimensional eigenvalue problem on $[0, R]$
\begin{equation}
		\label{4.2}
		\left \{
		\begin{aligned}
			& \bar{\mathfrak{L}}\varphi
			= -\lambda |\varphi|^{p-2} \varphi,\\
			&\varphi(0)=0,\ \varphi'(R)=0,
		\end{aligned}
		\right.
	\end{equation}
and denote by $\bar{\lambda}_1(m,p,\kappa_1,\kappa_2,D)$, written as $\bar{\lambda}_1$ for short,   the first eigenvalue of  problem \eqref{4.2}. Then it is easy to see that $\bar{\lambda}_1$ can be characterized by
	\[
	\bar{\lambda}_1
	=\inf
	\bigg\{
	\frac
	{ \int_0^{R} |\phi'|^p C_{\kappa_2,\Lambda}^{2m-2} C_{4\kappa_1,\Lambda} \ \mathrm{d} s }
	{ \int_0^{R} |\phi|^p C_{\kappa_2,\Lambda}^{2m-2} C_{4\kappa_1,\Lambda} \ \mathrm{d} s }
	\ \big|\phi \in W^{1,p}\big((0,R)\big)\backslash\{0\},\  \phi(0)=0
	\bigg\},
	\]
	where $C_{\kappa,\Lambda}$ is defined by \eqref{1.7}.
	
	\begin{proof}[Proof of Theorem \ref{thm 1.3}]
	    Let $\varphi$ be an eigenfunction corresponding to $\bar\lambda_1$, then
	    \[
		\bar{\mathfrak{L}}\varphi
		= -\bar{\lambda}_1 |\varphi|^{p-2} \varphi
		\]
	    with $\varphi(0)=0$ and $\varphi'(R)=0$.
	    Similarly as in the proof of Lemma \ref{lm existence of odd function}, we can choose $\varphi$ such that $\varphi(s) >0$ on $(0,R]$, and $\varphi'(s)>0$ on $[0, R)$. Define a trial function for $\lambda_{1,p}(M)$ by
	    \[
		v(x):=\varphi(d(x,\partial M)),
		\]
	  then Lemma \ref{lemma main dirichlet eigenvalue} gives
		\begin{align}\label{4.3}
		\Delta_p v(x)
		\leq
		\bar{\mathfrak{L}}\varphi(d(x,\partial M))
		\end{align}
away from the cut locus of $\p M$, and thus globally in the distributional sense, see \cite[Lemma 5.2]{LW21-b}.
	    Recall that $\varphi$ is an eigenfunction with respect to $\bar\lambda_1$, namely
	    \begin{align}\label{4.4}
		\bar{\mathfrak{L}}\varphi(d(x,\partial M))
		=
		-\bar{\lambda}_1 |v(x)|^{p-2} v(x),
		\end{align}
		thus we conclude from \eqref{4.3} and \eqref{4.4} that
	    \begin{align*}
		\Delta_p v
		\leq
		-\bar{\lambda}_1 |v|^{p-2} v, \quad x\in M.
	    \end{align*}
By the definition of $v(x)$, we see $v(x)>0$ for $x\in M$, and $v(x)=0$ for $x\in \p M$. Then using Barta's inequality (cf. \cite[Theorem 3.1]{LW21-b}), we have
		\[
		\lambda_{1,p}(M) \geq \bar{\lambda}_1,
		\]
		proving Theorem \ref{thm 1.3}.
	\end{proof}
	
Lastly, we emphasize that when $\kappa_1=\kappa_2=\k$,  a standard argument (see for instance \cite[Section 6]{LW21-b}) indicates that when $M$ is a $(\kappa,\Lambda)$-model space in the complex space forms, the first Dirichlet eigenfunction of the $p$-Laplacian can be written in the form $u=\varphi\circ d(x,\partial M)$, and $\vp$ is the first eigenfunction of one-dimensional problem \eqref{4.2}, which gives the equality case of Theorem \ref{thm 1.2}. Therefore estimate \eqref{1.9} is sharp when $\k_1=\k_2$.

	\bibliographystyle{plain}
	\bibliography{ref}

\begin{thebibliography}{10}

\bibitem{An15}
Ben Andrews.
\newblock Moduli of continuity, isoperimetric profiles, and multi-point
  estimates in geometric heat equations.
\newblock In {\em Surveys in differential geometry 2014. {R}egularity and
  evolution of nonlinear equations}, volume~19 of {\em Surv. Differ. Geom.},
  pages 1--47. Int. Press, Somerville, MA, 2015.

\bibitem{AC09}
Ben Andrews and Julie Clutterbuck.
\newblock Time-interior gradient estimates for quasilinear parabolic equations.
\newblock {\em Indiana Univ. Math. J.}, 58(1):351--380, 2009.

\bibitem{AC11}
Ben Andrews and Julie Clutterbuck.
\newblock Proof of the fundamental gap conjecture.
\newblock {\em J. Amer. Math. Soc.}, 24(3):899--916, 2011.

\bibitem{AC13}
Ben Andrews and Julie Clutterbuck.
\newblock Sharp modulus of continuity for parabolic equations on manifolds and
  lower bounds for the first eigenvalue.
\newblock {\em Anal. PDE}, 6(5):1013--1024, 2013.

\bibitem{BS19}
Casey Blacker and Shoo Seto.
\newblock First eigenvalue of the {$p$}-{L}aplacian on {K}\"{a}hler manifolds.
\newblock {\em Proc. Amer. Math. Soc.}, 147(5):2197--2206, 2019.

\bibitem{CW94}
Mu~Fa Chen and Feng~Yu Wang.
\newblock Application of coupling method to the first eigenvalue on manifold.
\newblock {\em Sci. China Ser. A}, 37(1):1--14, 1994.

\bibitem{HWZ20}
Chenxu He, Guofang Wei, and Qi~S. Zhang.
\newblock Fundamental gap of convex domains in the spheres.
\newblock {\em Amer. J. Math.}, 142(4):1161--1191, 2020.

\bibitem{Ka84}
Atsushi Kasue.
\newblock On a lower bound for the first eigenvalue of the {L}aplace operator
  on a {R}iemannian manifold.
\newblock {\em Ann. Sci. \'{E}cole Norm. Sup. (4)}, 17(1):31--44, 1984.

\bibitem{Kr92}
Pawel Kr\"{o}ger.
\newblock On the spectral gap for compact manifolds.
\newblock {\em J. Differential Geom.}, 36(2):315--330, 1992.

\bibitem{Li79}
Peter Li.
\newblock A lower bound for the first eigenvalue of the {L}aplacian on a
  compact manifold.
\newblock {\em Indiana Univ. Math. J.}, 28(6):1013--1019, 1979.

\bibitem{LW05}
Peter Li and Jiaping Wang.
\newblock Comparison theorem for {K}\"{a}hler manifolds and positivity of
  spectrum.
\newblock {\em J. Differential Geom.}, 69(1):43--74, 2005.

\bibitem{LY80}
Peter Li and Shing~Tung Yau.
\newblock Estimates of eigenvalues of a compact {R}iemannian manifold.
\newblock In {\em Geometry of the {L}aplace operator ({P}roc. {S}ympos. {P}ure
  {M}ath., {U}niv. {H}awaii, {H}onolulu, {H}awaii, 1979)}, Proc. Sympos. Pure
  Math., XXXVI, pages 205--239. Amer. Math. Soc., Providence, R.I., 1980.

\bibitem{LW21-b}
Xiaolong Li and Kui Wang.
\newblock First {R}obin eigenvalue of the {$p$}-{L}aplacian on {R}iemannian
  manifolds.
\newblock {\em Math. Z.}, 298(3-4):1033--1047, 2021.

\bibitem{LW21}
Xiaolong Li and Kui Wang.
\newblock Lower bounds for the first eigenvalue of the {L}aplacian on
  {K}\"{a}hler manifolds.
\newblock {\em Trans. Amer. Math. Soc.}, 374(11):8081--8099, 2021.

\bibitem{LW21-c}
Xiaolong Li and Kui Wang.
\newblock Sharp lower bound for the first eigenvalue of the weighted
  {$p$}-{L}aplacian {I}.
\newblock {\em J. Geom. Anal.}, 31(8):8686--8708, 2021.

\bibitem{LW21-d}
Xiaolong Li and Kui Wang.
\newblock Sharp lower bound for the first eigenvalue of the weighted
  $p$-laplacian {II}.
\newblock {\em Math. Res. Lett.}, 28(5):1459--1479, 2021.

\bibitem{Li58}
Andr\'{e} Lichnerowicz.
\newblock {\em G\'{e}om\'{e}trie des groupes de transformations}.
\newblock Travaux et Recherches Math\'{e}matiques, III. Dunod, Paris, 1958.

\bibitem{Mu09}
Ovidiu Munteanu.
\newblock A sharp estimate for the bottom of the spectrum of the {L}aplacian on
  {K}\"{a}hler manifolds.
\newblock {\em J. Differential Geom.}, 83(1):163--187, 2009.

\bibitem{NV14}
Aaron Naber and Daniele Valtorta.
\newblock Sharp estimates on the first eigenvalue of the {$p$}-{L}aplacian with
  negative {R}icci lower bound.
\newblock {\em Math. Z.}, 277(3-4):867--891, 2014.

\bibitem{Ni13}
Lei Ni.
\newblock Estimates on the modulus of expansion for vector fields solving
  nonlinear equations.
\newblock {\em J. Math. Pures Appl. (9)}, 99(1):1--16, 2013.

\bibitem{NZ18}
Lei Ni and Fangyang Zheng.
\newblock Comparison and vanishing theorems for {K}\"{a}hler manifolds.
\newblock {\em Calc. Var. Partial Differential Equations}, 57(6):Paper No. 151,
  31, 2018.

\bibitem{BS22}
Benjamin Rutkowski and Shoo Seto.
\newblock Explicit lower bound of the first eigenvalue of the laplacian on
  k\"{a}hler manifolds.
\newblock {\em arXiv:2207.10966}, 2022.

\bibitem{SWW19}
Shoo Seto, Lili Wang, and Guofang Wei.
\newblock Sharp fundamental gap estimate on convex domains of sphere.
\newblock {\em J. Differential Geom.}, 112(2):347--389, 2019.

\bibitem{Tsu57}
Y\^{o}tar\^{o} Tsukamoto.
\newblock On {K}\"{a}hlerian manifolds with positive holomorphic sectional
  curvature.
\newblock {\em Proc. Japan Acad.}, 33:333--335, 1957.

\bibitem{Va12}
Daniele Valtorta.
\newblock Sharp estimate on the first eigenvalue of the {$p$}-{L}aplacian.
\newblock {\em Nonlinear Anal.}, 75(13):4974--4994, 2012.

\bibitem{ZW17}
Yuntao Zhang and Kui Wang.
\newblock An alternative proof of lower bounds for the first eigenvalue on
  manifolds.
\newblock {\em Math. Nachr.}, 290(16):2708--2713, 2017.

\bibitem{ZY84}
Jia~Qing Zhong and Hong~Cang Yang.
\newblock On the estimate of the first eigenvalue of a compact {R}iemannian
  manifold.
\newblock {\em Sci. Sinica Ser. A}, 27(12):1265--1273, 1984.

\end{thebibliography}
	
\end{document}